\documentclass[12pt,reqno]{amsart}

\usepackage{amssymb}
\usepackage{ifthen}
\usepackage[pdftex]{hyperref}

\hypersetup{colorlinks=true, linkcolor=black, citecolor=black}

\newcommand{\showmorestuff}{no}    

\newcommand{\Z}{{\mathbb Z}}

\newcommand{\cS}{{\mathcal S}}

\renewcommand{\o}[1]{\overline{#1}}

\newcommand{\oa}{{\overline a}}
\newcommand{\ob}{{\overline b}}

\newcommand{\os}{{\overline s}}

\newcommand{\oA}{{\overline A}}
\newcommand{\oB}{{\overline B}}
\newcommand{\oF}{{\overline F}}
\newcommand{\oG}{{\overline G}}
\newcommand{\oH}{{\overline H}}
\newcommand{\oK}{{\overline K}}
\newcommand{\oS}{{\overline S}}

\newcommand{\dfc}{\mathsf{d}}
\newcommand{\Dfc}{\mathsf{D}}

\newcommand{\tc}{{\widetilde c}}

\newcommand{\tA}{{\widetilde A}}
\newcommand{\tG}{{\widetilde G}}
\newcommand{\tH}{{\widetilde H}}
\newcommand{\tK}{{\widetilde K}}
\newcommand{\tL}{{\widetilde L}}
\newcommand{\tP}{{\widetilde P}}

\newcommand{\alp}{\alpha}

\newcommand{\del}{\delta}
\newcommand{\eps}{\varepsilon}
\newcommand{\kap}{\varkappa}

\newcommand{\sig}{\sigma}
\renewcommand{\phi}{\varphi}

\newcommand{\Del}{\Delta}

\newcommand{\longc}{,\dotsc,}
\newcommand{\longp}{+\dotsb+}
\newcommand{\longe}{=\dotsb=}

\newcommand{\est}{\varnothing}
\newcommand{\sbs}{\subset}
\newcommand{\seq}{\subseteq}
\newcommand{\stm}{\setminus}

\newcommand{\<}{\langle}
\renewcommand{\>}{\rangle}

\newtheorem{claim}{Claim}
\newtheorem{lemma}{Lemma}
\newtheorem{theorem}{Theorem}
\newtheorem{corollary}{Corollary}
\newtheorem{proposition}{Proposition}
\theoremstyle{remark}
\newtheorem{remark}{Remark}
\newtheorem{example}{Example}

\newcommand{\refc}[1]{\ref{c:#1}}
\newcommand{\refl}[1]{\ref{l:#1}}
\newcommand{\refm}[1]{\ref{m:#1}}
\newcommand{\reft}[1]{\ref{t:#1}}
\newcommand{\refp}[1]{\ref{p:#1}}
\newcommand{\refr}[1]{\ref{r:#1}}
\newcommand{\refs}[1]{\ref{s:#1}}
\newcommand{\refb}[1]{\cite{b:#1}}
\newcommand{\refe}[1]{\eqref{e:#1}}

\title[Small doubling in moderate-torsion groups]%
  {Small doubling \\ in groups with moderate torsion}

\author{Vsevolod F. Lev}
\email{seva@math.haifa.ac.il}
\address{Department of mathematics, the University of Haifa at Oranim,
  Tivon 36006, Israel}

\begin{document}
\baselineskip=16pt

\begin{abstract}
We determine the structure of a finite subset $A$ of an abelian group given
that $|2A|<3(1-\eps)|A|,\ \eps>0$; namely, we show that $A$ is contained
either in a ``small'' one-dimensional coset progression, or in a union of
fewer than $\eps^{-1}$ cosets of a finite subgroup.

The bounds $3(1-\eps)|A|$ and $\eps^{-1}$ are best possible in the sense
that none of them can be relaxed without tightening another one, and the
estimate obtained for the size of the coset progression containing $A$ is
sharp.

In the case where the underlying group is infinite cyclic, our result
reduces to the well-known \emph{Freiman's $(3n-3)$-theorem}; the former
thus can be considered as an extension of the latter onto arbitrary abelian
groups, provided that there is ``not too much torsion involved''.
\end{abstract}

\maketitle

\section{Introduction and summary of results}\label{s:intro}

For subsets $A$ and $B$ of an additively written abelian group, by $A+B$ we
denote the set of all group elements representable as $a+b$ with $a\in A$ and
$b\in B$. We abbreviate $A+A$ as $2A$ and define the \emph{doubling
coefficient} of a finite, nonempty set $A$ to be the quotient $|2A|/|A|$.

It is a basic folklore fact that if $A$ is a finite set of integers, then
$|2A|\ge 2|A|-1$; more generally, if $A$ and $B$ are finite nonempty subsets
of a torsion-free abelian group, then $|A+B|\ge|A|+|B|-1$. An extension of
this fact onto general abelian groups with torsion is a deep result due to
Kneser, discussed in Section~\refs{kneser}.

In another direction, Freiman~\refb{f} has established the structure of
integer sets $A$ satisfying $|2A|\le 3|A|-3$; that is, roughly, sets with the
doubling coefficient up to $3$. This result, commonly referred to as
\emph{Freiman's $(3n-3)$-theorem}, along with its generalizations onto
distinct set summands, can be found in any standard additive combinatorics
monograph; see, for instance, \cite[Theorem~1.13]{b:n}.

It is a notoriously difficult open problem to merge together the results of
Kneser and Freiman establishing the structure of sets with the doubling
coefficient less than $3$ in abelian groups with torsion. This paper is
intended as a step towards the solution of this problem.

Our main result shows that a small-doubling set is either contained in the
union of a small number of cosets of a finite subgroup, or otherwise is
densely contained in a coset progression.
\begin{theorem}\label{t:big}
Let $A$ be a finite subset of an abelian group $G$ such that $A$ cannot be
covered with fewer than $n$ cosets of a finite subgroup of $G$, for some real
$n>0$. If $|2A|<3\big(1-\frac1n\big)|A|$, then there exist an arithmetic
progression $P\seq G$ of size $|P|\ge 3$ and a finite subgroup $K\le G$ such
that $|P+K|=|P||K|$, $A\seq P+K$, and $(|P|-1)|K|\le|2A|-|A|$.
\end{theorem}

\begin{remark}
The equality $|P+K|=|P||K|$ of Theorem~\reft{big} and Theorem~\reft{small}
below is an easy corollary of the other assertions: if it fails, then there
are two elements of $P$ that fall into the same $K$-coset; hence, $P+K$ is a
coset of a finite subgroup; therefore $n\le 1$, which is inconsistent with
the assumption $|2A|<3(1-1/n)|A|$. For this reason, we ignore the equality in
question and never get back to it till the end of the paper.
\end{remark}

\begin{remark}
Letting $\tau=|2A|/|A|$, the conclusion $(|P|-1)|K|\le|2A|-|A|$ can be
rewritten as $|P+K|\le(\tau-1)|A|+|K|$; the meaning of this estimate is that
$A$ is dense in $P+K$.
\end{remark}

We derive Theorem~\reft{big} from the following, essentially equivalent,
result.
\begin{theorem}\label{t:small}
Suppose that the abelian group $G$ has the direct sum decomposition
$G=\Z\oplus H$ with $H<G$ finite. Let $A\seq G$ be a finite set, and let $n$
be number of elements of the image of $A$ under the projection $G\to\Z$ along
$H$. If $|2A|<3\big(1-\frac1n\big)|A|$, then there exist an arithmetic
progression $P\seq G$ of size $3\le |P|\le(\tau-1)n+1$ and a subgroup $K\le
H$ such that $|P+K|=|P||K|$, $A\seq P+K$, and $(|P|-1)|K|\le|2A|-|A|$.
\end{theorem}

The equality $G=\Z\oplus H$ means that $G$ is the direct sum of its infinite
cyclic subgroup and the subgroup $H$. To simplify the notation, the former is
identified with the group of integers.

The following example shows that Theorem~\reft{small} is sharp in the sense
that the assumption $|2A|<3(1-\frac1n)|A|$ cannot be relaxed, and the
conclusion $(|P|-1)|K|\le|2A|-|A|$ cannot be strengthened.
\begin{example}\label{x:cylinder}
Let $P:=[0,l]$ and $A:=\big([0,n-2]\cup\{l\}\big)+K$, where $K\le H$, and
$l\ge n-1\ge 2$ are integers; thus, $|A|=n|K|$. If $l>2n-3$, then
$|2A|=(3n-3)|K|=3\big(1-n^{-1}\big)|A|$, and $A$ fails to have the structure
described in Theorem~\reft{bp} as $|2A|-|A|=(2n-3)|K|<(|P|-1)|K|$. Thus, to
conclude that a set $A\seq\Z\oplus H$ with $|2A|<3(1-\eps)|A|$ is densely
contained in a coset progression, one needs to assume that $A$ cannot be
covered with fewer than $\eps^{-1}$ cosets of a finite subgroup (or make some
other extra assumption).

On the other hand, if $l\le 2n-3$, then $|2A|=(l+n)|K|$; therefore,
$|2A|-|A|=(|P|-1)|K|$.
\end{example}

\begin{remark}\label{r:rem1}
The inequality $|P|\ge 3$ of Theorem~\reft{small} follows in fact
automatically from the other assertions of the theorem. We cannot have
$|P|=1$ because this would lead to $n=1$, and consequently to $|2A|<0$. We
also cannot have $|P|=2$ because this would result in $n=2$ and
$|2A|<\frac32\,|A|$. The latter, in its turn, is known to imply (see, for
instance, Lemma~\refl{1.5} below) that $A$ is contained in a coset of a
finite subgroup of $G$; hence, in an $H$-coset. This, however, contradicts
the equality $n=2$. The same applies to Theorem~\reft{big}.

Similarly, the upper-bound estimate $|P|\le(\tau-1)n+1$ in
Theorem~\reft{small} follows from
  $$ (\tau-1)|A| = |2A|-|A| \ge (|P|-1)|K| \ge (|P|-1)\,\frac{|A|}n, $$
and we thus can safely forget this estimate for the rest of the paper.
\end{remark}

\begin{remark}
In the particular case where $H$ is trivial, and $A$ is a subset of the
infinite cyclic group, Theorem~\reft{small} is equivalent to Freiman's
classical \emph{$(3n-3)$-theorem}, see~\refb{f} or \cite[Theorem~1.13]{b:n}.
Theorem~\reft{small} thus can be considered as an extension of the
$(3n-3)$-theorem onto the groups with torsion.
\end{remark}

\begin{remark}
As a corollary of Theorem~\reft{small}, for any finite set $A\seq\Z\oplus H$,
denoting by $n$ the size of the projection of $A$ onto $\Z$ along $H$, we
have $|2A|\ge(2-\frac1n)|A|$. This follows by letting $\tau:=|2A|/|A|$ and
observing that
  $$ (\tau-1)|A| = |2A|-|A| \ge (|P|-1)|K|
                                \ge (n-1)|K| \ge \Big(1-\frac1n\Big) |A|. $$
We remark that, while the resulting estimate $|2A|\ge(2-\frac1n)|A|$ may not
be completely trivial, it is not particularly deep either, and can be proved
independently of the theorem, with a simple combinatorial reasoning in the
spirit of the proof of Lemma~\refl{A0Al} (Section~\refs{basic}).
\end{remark}

It should be possible to use our method to treat sumsets of the form $A+B$
with the distinct set summands, and in particular to prove analogues of
Theorems~\reft{big} and~\reft{small} for the difference sets $A-A$.

Theorem~\reft{small} can be compared against the following result of
Balasubramanian and Pandey, which is an elaboration on an earlier result of
Deshouillers and Freiman~\cite[Theorem~2]{b:df}.
\begin{theorem}[Balasubramanian-Pandey~{\cite[Theorem~5]{b:bp}}]\label{t:bp}
Let $d\ge 2$ be an integer and suppose that $A\seq\Z\oplus(\Z/d\Z)$ is a
finite set with $|2A|<2.5|A|$. For $z\in\Z$, let $A_z:=A\cap(z+\Z/d\Z)$, and
let $B:=\{z\in\Z\colon A_z\ne\est\}$. If $|B|\ge 6$ and $\gcd(B-B)=1$, then
there exists a subgroup $K\le\Z/d\Z$ and elements $x,y\in\Z/d\Z$ such that,
letting $l:=\max B-\min B$, we have
\begin{itemize}
\item[i)]   $A\seq\{(b,bx+y)\colon b\in B\}+K$;
\item[ii)]  there exists $b\in B$ with $|A_b|\ge\frac23|K|$;
\item[iii)] $l|K|\le|2A|-|A|$.
\end{itemize}
\end{theorem}

Balasubramanian and Pandey also include into the statement the estimate
$l<\frac32|B|$, but in fact this estimate follows easily from i) and iii):
  $$ l|K| < 2.5|A|-|A| = \frac32\,|A| \le \frac32\,|B||K|. $$

In the same vein, i) and iii) imply an estimate which is only slightly weaker
than ii): namely, by iii) we have $l|K|\le(\tau-1)|A|$; therefore, by
averaging, there exists an element $b\in B$ with
  $$ |A_b| \ge \frac{|A|}{|B|}
         \ge \frac{(\tau-1)|A|}{(\tau-1)(l+1)}
                     \ge \frac l{l+1}\frac{|K|}{\tau-1}
                           \ge \frac23\,\Big(1-\frac{1}{l+1}\Big)\, |K|. $$
To match the Balasubramanian-Pandey estimate $\max_{b\in
B}|A_b|\ge\frac23|K|$, we prove in Section~\refs{deduct} the following
theorem showing (subject to Theorem~\reft{small}) that if $n$ is sufficiently
large, then there exists a $K$-coset containing at least $\frac{|K|}{\tau-1}$
elements of $A$.
\begin{theorem}\label{t:propo}
Suppose that $G$, $H$, $A$, $n$, $P$, and $K$ are as in Theorem~\reft{small},
and let $\tau:=|2A|/|A|$. If $n\ge\frac{4\tau-6}{(\tau-2)(3-\tau)}$, then
there exists a $K$-coset containing at least $\frac{|K|}{\tau-1}$ elements of
$A$.
\end{theorem}

Compared to Theorem~\reft{bp}, our Theorem~\reft{small} applies to the groups
$\Z\oplus H$ with $H$ not necessarily cyclic and, most importantly, allows
the doubling coefficient to be as large as $3-o(1)$ (instead of $2.5$), which
is best possible, as shown above.

The layout of the remaining part of the paper is as follows. In
Section~\refs{deduct} we deduce Theorems~\reft{big} and~\reft{propo} from
Theorem~\reft{small}, allowing us to concentrate on the proof of the latter
theorem for the rest of the paper. In Section \refs{general} we collect some
general results needed for the proof; in particular, we introduce and briefly
discuss Kneser's theorem. In Section~\refs{basic} we prove some basic
estimates related to the particular settings of Theorem~\reft{small} (in
contrast with Section~\refs{general} where the results are of general
nature). Section~\refs{twocoset} contains a result which, essentially,
establishes the special case of Theorem~\reft{small} where the set $A$ can be
partitioned into two ``additively independent'' subsets. Finally, we prove
Theorem~\reft{small} in Section~\refs{proof}.

\bigskip
\ifthenelse{\equal{\showmorestuff}{yes}}%
{{\color{brown}%
\emph{The applications to consider:
\begin{itemize}
\item[--] The structure of small-doubling sets in $\Z_m$ (which was the
    original motivation of Deshouillers-Freiman and
    Balasubramanian-Pandey);
\item[--] as an application of the first application: the structure of sets
    in $\Z$ with doubling up to $10/3$ (see my unpublished manuscript);
\item[--] as an application of the second application: improving Freiman's
    2.4-theorem.
\end{itemize}}}
}{} 

\section{Deduction of Theorems~\reft{big} and~\reft{propo} from
  Theorem~\reft{small}}\label{s:deduct}

\begin{proof}[Proof of Theorem~\reft{big}]
Let $A$ be a finite subset of an abelian group $G$ such that $A$ cannot be
covered with fewer than $n$ cosets of a finite subgroup of $G$, while
\begin{equation}\label{e:sda-m}
  |2A| < 3\Big(1-\frac1n\Big)|A|,
\end{equation}
with a real $n>0$. We want to prove, assuming Theorem~\reft{small}, that
there exist an arithmetic progression $P\seq G$ and a subgroup $K\le G$ such
that $A\seq P+K$ and $(|P|-1)|K|\le|2A|-|A|$. As explained in
Section~\refs{intro} (Remark~\refr{rem1}), the progression will satisfy
$|P|\ge 3$ and $|P+K|=|P||K|$.

Without loss of generality, we assume that $G$ is generated by $A$. By the
fundamental theorem of finitely generated abelian groups, there is then an
integer $r\ge 0$ and a finite subgroup $H\le G$ such that $G\cong\Z^r\oplus
H$. Indeed, we have $r\ge 1$ as otherwise $G$ would be finite; hence, $A$
would be contained in just one single finite coset (the group $G$ itself),
forcing $n\le 1$ and thus contradicting the small-doubling
assumption~\refe{sda-m}.

Let $G':=\Z\oplus H$. To avoid confusion, throughout the proof we use the
direct product notation for the elements of the groups $G$ and $G'$.

Fix an integer $M>0$ divisible by all positive integers up to $|2A|$, and
consider the mapping $\psi\colon G\to G'$ defined by
  $$ \psi(x_1\longc x_r,h) := (x_1+Mx_2\longp M^{r-1}x_r,h);
                                     \quad x_1\longc x_r\in\Z,\ h\in H. $$
If $M$ is large enough (as we assume below), then different elements of $A$
have different images under $\psi$, and similarly for $2A$; consequently,
writing $A':=\psi(A)$, we have $|A'|=|A|$ and $|2A'|=|2A|$, whence
  $$ |2A'| < 3\Big(1-\frac1n\Big)|A'| $$
(we implicitly use here the equality $2\psi(A)=\psi(2A)$).

Denote by $m$ the number of elements of the projection of $A$ onto the first
(torsion-free) component of $G$. If $M$ is sufficiently large, then this is
also the number of elements of the projection of $A'$ onto the first
component of $G'$. Since $A$ is not contained in a union of fewer than $n$
cosets, we have $m\ge n$, resulting in
  $$  |2A'| < 3\Big(1-\frac1m\Big)|A'|. $$
Applying Theorem~\reft{small}, we conclude that there exist a finite
arithmetic progression $P'\seq G'$ and a subgroup $K\le H$ such that $A'\seq
P'+K$ and
\begin{equation}\label{e:0814a}
  (|P'|-1)|K|\le|2A'|-|A'|=|2A|-|A|.
\end{equation}
We assume that $P'$ is the shortest progression possible with $A'\seq P'+K$.

Write $N:=|P'|-1$, and let $c\in G'$ and $(d,h)\in G'$ denote the initial
term and the difference of the progression $P'$, respectively; thus,
  $$ P' = c + \{ j(d,h) \colon j\in[0,N]\};\quad d\in\Z,\ h\in H. $$
Notice that $d\ne 0$, as otherwise we would have $A'\seq P'+K\seq c+H$, as a
result of which $A'$, and therefore also $A$, would be contained in a single
$H$-coset.

Since $P'$ is the shortest possible progression with $A'\seq P'+K$, there are
elements
 $(a_1\longc a_r,f),(b_1\longc b_r,g)\in A$ such that
$\psi(a_1\longc a_r,f)=c$ and $\psi(b_1\longc b_r,g)=c+N(d,h)$; consequently,
  $$ (b_1-a_1)+M(b_2-a_2)\longp M^{r-1}(b_r-a_r) = Nd. $$
Since $N=|P'|-1\le|2A|-|A|<|2A|$, and recalling that $M$ was chosen to be
divisible by all integers up to $|2A|$, we have $N\mid M$, and therefore
$b_1-a_1$ is a multiple of $N$. Thus
\begin{equation}\label{e:1410a}
  d=(b_1-a_1)N^{-1}+MN^{-1}(b_2-a_2)\longp M^{r-1}N^{-1}(b_r-a_r),
\end{equation}
where all summands in the right-hand side are integers.

We know that for any element $(\alp_1\longc\alp_r,\eta)\in A$, there exist
$j\in[0,N]$ and $k\in K$ such that
\begin{align*}
  (\alp_1\longp M^{r-1}\alp_r,\eta)
    &= c+j(d,h)+(0,k) \\
    &= (a_1\longp M^{r-1}a_r,f)+j(d,h)+(0,k).
\end{align*}
Recalling~\refe{1410a}, we obtain
  $$ (\alp_1-a_1)\longp M^{r-1}(\alp_r-a_r) = jd
      = j(b_1-a_1)N^{-1}\longp jM^{r-1}N^{-1}(b_r-a_r); $$
that is,
\begin{equation}\label{e:0814b}
  (\alp_1-a_1)N\longp M^{r-1}(\alp_r-a_r)N
                                   = j(b_1-a_1)\longp jM^{r-1}(b_r-a_r)
\end{equation}
with $j\in[0,N]$ depending on $\alp_1\longc\alp_r$. (Notice that $N$ depends
on $M$, but is bounded: $N\le N|K|\le|2A|-|A|$ by~\refe{0814a}.) Choosing $M$
sufficiently large, from~\refe{0814b} we get
\begin{equation}\label{e:1310a}
  (\alp_i-a_i)N = j\,(b_i-a_i),\ 1\le i\le r,
\end{equation}
showing that $(b_i-a_i)j$ is divisible by $N$. Using again the fact that $P'$
is the shortest possible progression with $A'\seq P'+K$, we conclude that the
possible values of $j$ that can emerge from different elements
$(\alp_1\longc\alp_r,\eta)\in A$ are coprime. Hence, there is a linear
combination of these values, with integer coefficients, which is equal to
$1$. Consequently, from~\refe{1310a}, all numbers $(b_i-a_i)N^{-1}$, $1\le
i\le r$, are integers, and then, by~\refe{1310a} again,
\begin{multline*}
  (\alp_1\longc\alp_r,\eta)
      = (a_1\longc a_r,f) +  j((b_1-a_1)N^{-1}\longc(b_r-a_r)N^{-1},h)
     + (0\longc 0,k).
\end{multline*}
This shows that $A\seq P+K$, where $P\seq G$ is the $(N+1)$-term arithmetic
progression with the initial term $(a_1\longc a_r,f)$ and the difference
$((b_1-a_1)N^{-1}\longc(b_r-a_r)N^{-1},h)$. Finally, by~\refe{0814a},
  $$ |2A|-|A|=|2A'|-|A'|\ge(|P'|-1)|K|=(|P|-1)|K|. $$
\end{proof}

\begin{proof}[Proof of Theorem~\reft{propo}]
Let $B$ denote the projection of $A$ onto $\Z$ along $H$; thus, $|P|\ge
|B|=n$, with equality if and only if $B$ is an arithmetic progression. If
this is not the case, then $|P|\ge n+1$ and, by averaging, there is a
$K$-coset containing at least
  $$ \frac{|A|}n \ge \frac{|A|}{|P|-1} \ge \frac{|K|}{\tau-1} $$
elements of $A$ (the last inequality following directly from the estimate
$|2A|-|A|\ge(|P|-1)|K|$ of Theorem~\reft{small}). Suppose thus that $B$ is an
arithmetic progression and, consequently, $|P|=n$ and $|2A|-|A|\ge(n-1)|K|$,
whence
  $$ |A|\ge\frac{n-1}{\tau-1}\,|K|. $$
Let
\begin{gather*}
  M := \max\{|A\cap(g+K)|\colon g\in P\}, \ \mu:=|M|/|K|, \\
  P_0 := \big\{ g\in P\colon |A\cap(g+K)| \le \frac12\,|K| \big\},
    \ P_1:=P\stm P_0,\ \text{and}\ m:=|P_0|.
\end{gather*}
Notice that $M>\frac12\,|K|$ as otherwise we would have
  $$ \frac12\,|K| \ge M \ge \frac{|A|}n
                                  \ge \Big(1-\frac1n\Big)\frac{|K|}{\tau-1}, $$
which is easily seen to contradict $\tau<3(1-\frac1n)$. Therefore $P_1$ is
nonempty, and $m<n$.

We want to show that $\mu>\frac1{\tau-1}$. Suppose for a contradiction that
this is wrong. Since $P+K$ is a union of $n$ pairwise disjoint $K$-cosets, of
which $m$ contain at most $\frac12|K|$ elements of $A$, and the remaining
$n-m$ contain at most $M$ elements each, we have
\begin{equation}\label{e:1801b}
  \frac{n-1}{\tau-1}\,|K| \le |A| \le m\cdot \frac12|K| + (n-m)\cdot M,
\end{equation}
leading to
  $$ \frac{n-1}{\tau-1}
          \le \frac12\,m  + (n-m)\mu < \frac12\,m  + \frac{n-m}{\tau-1}, $$
where the last inequality follows from the assumption $\mu<1/(\tau-1)$. This
simplifies to the estimate
\begin{equation}\label{e:1801a}
  m < \frac2{3-\tau}
\end{equation}
which we will need shortly.

The set $2P_1+K$ is a union of $|2P_1|\ge 2|P_1|-1=2(n-m)-1$ distinct
$K$-cosets contained in $2A$ by the pigeonhole principle. The set $P+P_1+K$
is a union of $|P+P_1|\ge|P|+|P_1|-1=2n-m-1$ distinct $K$-cosets, each of
them containing at least $\frac12\,|K|$ elements of $2A$. We thus can find
$2n-2m-1$ cosets represented by the elements of $2P_1$, and then $m$ more
cosets represented by the elements of $P+P_1$. Altogether, we get $2n-m-1$
cosets containing at least
  $$ (2n-2m-1)|K|+\frac12\,|K|m = \Big(2n-\frac32\,m-1\Big)|K| $$
elements of $2A$. It follows that
  $$ 2n-\frac32\,m-1 \le \frac{|2A|}{|K|} = \tau \frac{|A|}{|K|}
        \le \Big(\frac12 m + (n-m)\mu\Big)\tau
            < \Big(\frac12 m + \frac{n-m}{\tau-1} \Big)\tau, $$
cf.~\refe{1801b}. Rearranging the terms gives
  $$ \Big(1-\frac1{\tau-1}\Big) n
        < \Big( \frac32+\frac{\tau}2-\frac{\tau}{\tau-1} \Big) m+1; $$
that is, using~\refe{1801a},
  $$ \frac{\tau-2}{\tau-1}\,n < \frac{\tau^2-3}{2(\tau-1)}m + 1
        < \frac{\tau^2-3}{(\tau-1)(3-\tau)} + 1 $$
leading to
  $$ n < \frac{\tau^2-3}{(\tau-2)(3-\tau)} + \frac{\tau-1}{\tau-2}
               = \frac{4\tau-6}{(\tau-2)(3-\tau)}, $$
and the assertion follows.
\end{proof}

The rest of the paper is devoted to the proof of Theorem~\reft{small}.

\section{General results}\label{s:general}\label{s:kneser}
In this section we collect some general results valid in any abelian group,
regardless of the particular settings of Theorem~\reft{small}.

For a subset $S$ of an abelian group, let $\pi(S)$ denote the \emph{period
(stabilizer)} of $S$; that is, $\pi(S)$ is the subgroup consisting of all
those group elements $g$ with $S+g=S$. The set $S$ is called \emph{aperiodic}
or \emph{periodic} according to whether $\pi(S)$ is or is not the zero
subgroup.

We start with a basic theorem due to Kneser which is heavily used in our
argument.
\begin{theorem}[{Kneser, \cite{b:kn1,b:kn2}; see also
  \cite[Theorem~4.1]{b:n}}]\label{t:Kneser}
If $B$ and $C$ are finite, non-empty subsets of an abelian group with
  $$ |B+C| \le |B|+|C|-1, $$
then letting $L:=\pi(B+C)$ we have
  $$ |B+C| = |B+L|+|C+L|-|L|. $$
\end{theorem}

We will be referring Theorem~\reft{Kneser} as \emph{Kneser's theorem}.

Since, in the above notation, we have $|B+L|\ge|B|$ and $|C+L|\ge|C|$,
Kneser's theorem shows that $|B+C|\ge|B|+|C|-|L|$, leading to

\begin{corollary}
If $B$ and $C$ are finite, non-empty subsets of an abelian group, such that
$|B+C|<|B|+|C|-1$, then $B+C$ is periodic.
\end{corollary}

The following lemma is well known, but tracing it down to the origin is
hardly possible.
\begin{lemma}\label{l:1.5}
Let $B$ be a finite subset of an abelian group. If $|2B|<\frac32\,|B|$, then
there is a subgroup $L$ such that $B-B=L$, and $2B$ is an $L$-coset (as a
result of which $B$ is contained in a unique $L$-coset).
\end{lemma}

We give a somewhat nonstandard, self-contained proof of the lemma.
\begin{proof}[Proof of Lemma~\refl{1.5}]
For a group element $g$, denote by $r(g)$ the number of representations of
$g$ as a difference of two elements of $B$. If $g\in B-B$, then choosing
arbitrarily $b,c\in B$ with $g=b-c$ we get
  $$ r(g) = |(b+B)\cap(c+B)| \ge 2|B| - |2B| > \frac12\,|B|. $$
By the pigeonhole principle, for any $g_1,g_2\in B-B$ there are
representations $g_1=b_1-c_1,\ g_2=b_2-c_2$ with $c_1=c_2$; consequently,
 $g_1-g_2=b_1-b_2\in B-B$, showing that $L:=B-B$ is a subgroup. Clearly, $B$
is contained in a unique $L$-coset.

As we have shown, for every element $g\in L=B-B$ we have $r(g)>\frac12\,|B|$.
As a result,
  $$ |B|(|B|-1) = \sum_{g\in L\stm\{0\}}r(g) > \frac12\,|B|\cdot(|L|-1), $$
implying $|B|>\frac12\,|L|$. Recalling that $B$ is contained in a unique
$L$-coset, and using the pigeonhole principle again, we conclude that $2B$ is
an $L$-coset.
\end{proof}

\begin{lemma}\label{l:pairs}
Suppose that $B$ is a subset of an abelian group with $0\in B$. If $N\ge 3$
is an integer such that $|B|=N+1$ and $|2B|=2N+1$ (thus $|2B\stm B|=N$), then
one of the following holds:
\begin{itemize}
\item[i)] there exist $b_1\longc c_N\in B$ such that $2B\stm
    B=\{b_1+c_1\longc b_N+c_N\}$, and every element of $B$ appears among
    $b_1\longc c_N$ at most $N$ times;
\item[ii)] there is a subgroup $L$ with $|L|=N$ and a group element $g$
    with $2g\notin L$ such that $B=L\cup\{g\}$. (In this case there exist
    $b_1\longc c_N\in B$ such that $2B\stm B=\{b_1+c_1\longc b_N+c_N\}$,
    and every element of $B$ appears among $b_1\longc c_N$ exactly once,
    except that $0$ does not appear at all, and $g$ appears $N+1$ times.)
\item[iii)] $N=2$ and there is a subgroup $L$ with $|L|=2$ and a group
    element $g$ with $2g\notin L$ such that $B=(g+L)\cup\{0\}$.
\end{itemize}
\end{lemma}

\begin{proof}
Leaving the case $N=2$ to the reader (hint: write $B=\{0,b,g\}$ and consider
two cases: $b+g=0$ and $b+g\ne 0$), we confine ourself to the general case
where $N\ge 3$.

Choose $b_1\longc c_N\in B$ arbitrarily to have
 $2B\stm B=\{b_1+c_1\longc b_N+c_N\}$. Since all sums $b_i+c_i$ are distinct,
for any $g\in B$ there is at most one index $i\in[1,N]$ with $b_i=c_i=g$.
Consequently, if there is an element $g\in B$ which appears at least $N+1$
times among $b_1\longc c_N$ (as we now assume), then in fact it appears $N+1$
times exactly: namely, $b_i=c_i=g$ for some $i\in[1,N]$ and, besides, for any
$j\neq i$, exactly one of $b_i$ and $c_j$ is equal to $g$. Redenoting, we
assume that $b_1=c_1\longe c_N=g$.

Notice that $2g=b_1+c_1\in 2B\stm B$ along with $0\in B$ show that $g\ne 0$.
Write $B_0:=B\stm\{0\}$ and $B_g:=B\stm\{g\}$. Since the sums $b_i+c_i=b_i+g$
are pairwise distinct, so are the elements $b_1\longc b_N\in B$. Moreover,
$b_1\longc b_N$ are nonzero in view of $b_i+g=b_i+c_i\notin B$ and $g\in B$,
and since $|B_0|=N$, it follows that $\{b_1\longc b_N\}=B_0$; consequently,
$2B\stm B=g+B_0$.

If there exist some $b,c\in B_g$ with $b+c\notin B$, then choosing
$i\in[1,N]$ with $b_i+g=b+c$ and replacing $b_i$ with $b$ and $c_i$ with $c$
in the $2N$-tuple $(b_1\longc c_N)$, we get another $2N$-tuple
 $(b_1'\longc c_N')$ such that the sums $b_i'+c_i'$ list all elements of
$2B\stm B$. If $i\in[2,N]$, then $g$ appears exactly $N$ times among
$b_1'\longc c_N'$, so that no other element of $B$ can appear $N+1$ or more
times. Similarly, if $i=1$, then in view of $c_2'\longe c_N'=g$, and since
all sums $b_2'+c_2'\longc b_N'+c_N'$ are pairwise distinct, every element
$b\in B_g$ appears at most $3<N+1$ times among $b_1'\longc c_N'$. Thus, the
assertion holds true in this case.

Suppose therefore that $b,c\in B_g$ with $b+c\notin B$ do not exist; that is,
$2B_g\seq B$. This gives $|2B_g|\le|B_g|+1$; hence, by Lemma~\refl{1.5} and
in view of $0\in B_g$, the set $L:=B_g-B_g=2B_g$ is a subgroup. Furthermore,
since $B_g\seq 2B_g=L$ and $|B_g|\ge|2B_g|-1=|L|-1$, we have either $B_g=L$,
or $B_g=L\stm\{l\}$ with some $l\in L$, $l\ne 0$. The latter case is in fact
impossible as we would have $l\in 2B_g\stm B$ in this case, contradicting the
present assumption $2B_g\seq B$. In the former case we have $B=L\cup\{g\}$
and $2B=L\cup(g+L)\cup\{2g\}$, with $2g\notin L$ in view of
$|2B|=2N+1=2|B|-1=2|L|+1$.
\end{proof}

\begin{lemma}\label{l:2101}
Suppose that $L$ is a subgroup, and that $B$ and $C$ are subsets of an
abelian group. Let $\phi_L$ denote the canonical homomorphism onto the
quotient group.
\begin{itemize}
\item[i)] We have $\phi_L(B\cup C)=\phi_L(B)\cup\phi_L(C)$.
\item[ii)] If at least one of $B+L=B$ and $C+L=C$ holds, then $\phi_L(B\cap
    C)=\phi_L(B)\cap\phi_L(C)$, and hence $\phi_L(B\stm
    C)=\phi_L(B)\stm\phi_L(C)$.
\end{itemize}
\end{lemma}

\begin{proof}
The first assertion is trivial and is stated for completeness only.

For the second assertion, we assume, for definiteness, that $B+L=B$, and show
that $\phi_L(B)\cap\phi_L(C)\seq\phi_L(B\cap C)$; the opposite inclusion is
trivial. Fix an element $t\in\phi_L(B)\cap\phi_L(C)$. Since $t\in\phi_L(C)$,
there exists $c\in C$ with $t=\phi_L(c)$. Now $\phi_L(c)=t\in\phi_L(B)$ gives
$c\in B+L=B$, showing that $c\in B\cap C$ and therefore
$t=\phi_L(c)\in\phi_L(B\cap C)$. The assertion follows.
\end{proof}

\begin{corollary}\label{c:FGHDel}
Suppose that $L$ is a subgroup, and that $B$ and $C$ are subsets of an
abelian group. If $B+L=B$, then $\phi_L(B\cap C)=\phi_L(B)$ is equivalent to
any of $\phi_L(B)\seq\phi_L(C)$ and $B+L\seq C$.
\end{corollary}

\begin{proof}
Applying the lemma, we get $\phi_L(B\cap C)=\phi_L(B)\cap\phi_L(C)$. Thus,
$\phi_L(B\cap C)=\phi_L(B)$ is equivalent to $\phi_L(B)\seq\phi_L(C)$, which
is immediately seen to be equivalent to $B\seq C+L$.
\end{proof}

\begin{lemma}\label{l:module}
If $G$ is an abelian group having the direct sum decomposition $G=\Z\oplus H$
with $H$ finite, then every subgroup $G'<G$ is of the form $G'=\<g\>+K$ with
some $g\in G$ and $K\le H$; indeed, one can take $K:=G'\cap H$.
\end{lemma}

\begin{proof}
The assertion is immediate if $G'\le H$; assume therefore that $G'\not\le H$.
In this case the projection of $G'$ onto $\Z$ along $H$ is a non-zero
subgroup of $\Z$; let $z'$ be its generator. For $k\in\Z$, the ``slice''
$G'(k):=G'\cap(k+H)$ is non-empty if and only if $z'\mid k$. Furthermore, for
any $k_1,k_2$ divisible by $z'$, and any fixed $d\in G'(k_2-k_1)$, we have
$G'(k_1)+d\seq G'(k_2)$. This shows that all slices $G'(k)$ with $k$
divisible by $z'$ are actually translates of each other; hence, each of them
is a coset of the subgroup $K:=G'(0)\le H$.

Fix arbitrarily $g\in G'(z')$. For any integer $k$ divisible by $z'$, we have
$(k/z')g\in G'\cap(k+H)=G'(k)$. It follows that $G'(k)=(k/z')g+K$ for any
integer $k$ with $z'\mid k$. As a result, $G'=\<g\>+K$.
\end{proof}

We need the following lemma in the spirit of~\refb{bp}.
\begin{lemma}\label{l:bpBC}
Suppose that $B$ and $C$ are finite, nonempty integer sets, and write
$m:=|B|$ and $B=\{b_1\longc b_m\}$, where the elements of $B$ are numbered in
an arbitrary order. Then there exist $c_2\longc c_m\in C$ such that the sums
$b_2+c_2\longc b_m+c_m$ are distinct from each other and from the elements of
the set $b_1+C$.
\end{lemma}

\begin{proof}
The proof follows the line of reasoning of~\refb{bp}.

Let $n:=|C|$ and consider the family of $m+n-1$ sets
\begin{align*}
  & b_1+C\longc b_1+C\quad\text{($n$ sets)} \\
  & b_2+C\longc b_m+C\quad\text{($m-1$ sets)}.
\end{align*}
Following Balasubramanian and Pandey, we use the Hall marriage theorem to
show that this set family has a system of distinct representatives; clearly,
this will imply the result.

Suppose thus that for some $1\le k\le m+n-1$ we are given a subsystem $\cS$
of $k$ sets from among those listed above, and show, to verify the hypothesis
of Hall's theorem, that $|\cup_{S\in\cS} S|\ge k$. Let $B'\seq B$ consist of
all those elements $b_i\in B\ (1\le i\le m)$ such that at least one of the
sets in $\cS$ has the form $b_i+C$. Then $\cup_{S\in\cS}=B'+C$ and we thus
want to show that $|B'+C|\ge k$. Since $|B'+C|\ge|B'|+|C|-1$, it suffices to
show that $|B'|+n-1\ge k$. Indeed, this inequality is trivial for $k\le n$,
while for $k\ge n$ it becomes evident upon writing $k=n+\kap$, $\kap\ge 0$
and observing any $n+\kap$ sets under consideration determine at least
$\kap+1=k-n+1$ elements $b_i$.
\end{proof}

\begin{corollary}\label{c:BPdist}
Suppose that the abelian group $G$ has the direct sum decomposition
$G=\Z\oplus H$ with $H<G$ finite. Let $B,C$ be finite, nonempty subsets of
$G$. If $m$ and $n$ denote the sizes of the images of $B$ and $C$,
respectively, under the projection $G\to\Z$ along $H$, then
  $$ |B+C| \ge \Big(1+\frac{n-1}m\Big)\,|B|. $$
\end{corollary}

\begin{proof}
Denote by $\psi$ the projection in question, and write $\psi(B):=\{b_1\longc
b_{m}\}$, where $b_1$ is chosen so that $|\psi^{-1}(b_1)\cap B|\ge |B|/m$;
otherwise, the elements of $\psi(B)$ are numbered arbitrarily. Let
$B_i:=\psi^{-1}(b_i)\cap B\ (1\le i\le m)$. By Lemma~\refl{bpBC} as applied
to the sets $\psi(B)$ and $\psi(C)$, there are (not necessarily distinct)
elements $c_2\longc c_m\in\psi(C)$ such that all sums $b_2+c_2\longc b_m+c_m$
are distinct from each other and from the elements of the set $b_1+\psi(C)$.
Consequently, the sumsets $B_2+(\psi^{-1}(c_2)\cap C)\longc
B_m+(\psi^{-1}(c_m)\cap C)$ are pairwise disjoint, and they are also disjoint
from each of the $n$ sumsets $B_1+(\psi^{-1}(c)\cap C),\ c\in\psi(C)$. As a
result,
\begin{multline*}
  |B+C| \ge \sum_{i=2}^m |B_i+(\psi^{-1}(c_i)\cap C)|
           + \sum_{c\in\psi(C)}|B_1+(\psi^{-1}(c)\cap C)| \\
         \ge |B_2|\longp|B_m| + n|B_1|
            = |B|+(n-1)|B_1| \ge |B| + \frac{n-1}{m}\,|B|.
\end{multline*}
\end{proof}

\section{Basic Estimates}\label{s:basic}

We collect in this section several basic estimates used in the proof of
Theorem~\reft{small}.

Suppose that $A$ is a finite subset of the group $G=\Z\oplus H$, where $H<G$
is finite abelian. For each $z\in\Z$, let $A_z:=A\cap(z+H)$, and write
$B:=\{z\in\Z\colon A_z\ne\est\}$; that is, $B$ is the image of $A$ under the
projection of $G$ onto $\Z$ along $H$. Suppose, furthermore, that $\min B=0$,
$\max B=l>0$, $0\in A_0$, and $\del\in A_l$. Finally, write $n:=|B|$,
$\sig:=|A_0|+|A_l|$, and $A^*:=A_0\cap(A_l-\del)$.
\begin{lemma}\label{l:A0Al}
We have $|2A|+|A^\ast|\ge\sig n$.
\end{lemma}

\begin{proof}
Considering the projections of the ``slices'' $A_b$ onto $\Z$, we get
\begin{align*}
  |2A| &\ge \sum_{\substack{z\in B \\ z<l}} |A_0+A_z|
           + |A_0+A_l| + \sum_{\substack{z\in B \\ z>0}}|A_z+A_l| \\
       &\ge (n-1)|A_0|+|A_0+A_l| + (n-1)|A_l|.
\end{align*}
To estimate the sum $A_0+A_l$ we notice that both $A_0+\del$ and $A_l$ are
subsets of $A_0+A_l$, whence
  $$ |A_0+A_l| \ge|(A_0+\del)\cup A_l|=(|A_0|+|A_l|)-|A^*|. $$
Combining these estimates yields the sought inequality.
\end{proof}

\begin{corollary}\label{c:3A2A}
Let $\tau:=|2A|/|A|$. If $\tau<3(1-\frac1n)$, then
\begin{gather}
  (3-\tau)(\tau |A|+|A^*|) > 3\sig, \label{e:0301b} \\
  3|A|-|2A| > \sig, \label{e:1001a}
\intertext{and}
  |2A| < 3|A| - 2|A^*|. \label{e:1912}
\end{gather}
\end{corollary}

\begin{proof}
To prove~\refe{0301b}, we multiply the inequality of the lemma by the
inequality $3-\tau>\frac3n$ following from $\tau<3(1-\frac1n)$, and then
substitute $|2A|=\tau|A|$.

For~\refe{1001a}, we use~\refe{0301b} to get
\begin{multline*}
  3|A|-|2A| = (3-\tau)|A| > \frac1\tau\,(3\sig-(3-\tau)|A^*|)
     = \frac3\tau\,\sig-\left(\frac3\tau-1\right)|A^*| \\
      \ge \frac3\tau\,\sig-\left(\frac3\tau-1\right)\cdot\frac\sig2
          = \frac12\left(\frac3\tau+1\right)\sig > \sig.
\end{multline*}
Finally,~\refe{1912} follows from
  $$ |2A| < 3|A| - \sig \le 3|A| - 2|A^*|. $$
\end{proof}

\section{The two-coset case}\label{s:twocoset}

In this section we prove a result which is easily seen to imply the special
case of Theorem~\reft{small} where the set $A$ is contained in a union of two
cosets of a subgroup $F<G$ (but not contained in a single coset of either $F$
or a subgroup containing $F$ as an index-$2$ subgroup).

\begin{proposition}\label{p:N=1}
Suppose that the abelian group $G$ has the direct sum decomposition
$G=\Z\oplus H$ with $H<G$ finite. Let $A_1,A_2\sbs G$ be finite, nonempty
subsets of $G$, and for $i\in\{1,2\}$ let $n_i:=|\psi(A_i)|$, where
$\psi\colon G\to\Z$ is the projection along $H$. Then
  $$  |2A_1|+|A_1+A_2|+|2A_2|
                         \ge 3\Big(1-\frac1{n_1+n_2}\Big)(|A_1|+|A_2|). $$
\end{proposition}

\begin{example}
If, for $i\in\{1,2\}$, we let $A_i=P_i+K$, where $P_i$ are arithmetic
progressions with the same difference not contained in $H$, and where $K\le
H$, then $n_i=|P_i|$ and
  $$ |2A_1|+|A_1+A_2|+|2A_2| = 3(|P_1|+|P_2|-1)|K|
        = 3\big(1-\frac1{n_1+n_2}\big)(|A_1|+|A_2|). $$
This shows that the estimate of the proposition is best possible.
\end{example}

\begin{proof}[Proof of Proposition~\refp{N=1}]
Recall that for a subset $S$ of an abelian group, by $\pi(S)$ we denote the
period of $S$; see Section~\refs{kneser}.

For $i\in\{1,2\}$, we have $\pi(2A_i)\le H$ (as $2A_i$ are finite), and
$|\psi(2A_i)|\ge 2n_i-1$, whence
  $$ |2A_i| \ge (2n_i-1)\,|\pi(2A_i)|. $$
On the other hand, by Kneser's theorem (Section~\refs{kneser})
  $$ |2A_i| \ge 2|A_i|-|\pi(2A_i)|. $$
Multiplying the latter inequality by $2n_i-1$ and adding the former to the
result (to cancel out the term $|\pi(2A_i)|$) we get
\begin{equation}\label{e:0410}
 |2A_i|\ge \Big(2-\frac1{n_i}\Big)|A_i|.
\end{equation}
Similarly, letting $n:=n_1+n_2$ and observing that
  $$ |\psi(A_1+A_2)| = |\psi(A_1)+\psi(A_2)| \ge n_1+n_2-1 = n-1, $$
we get $|A_1+A_2|\ge(n-1)|\pi(A_1+A_2)|$ and
$|A_1+A_2|\ge|A_1|+|A_2|-|\pi(A_1+A_2)|$, implying
\begin{equation}\label{e:sumA1A2}
  |A_1+A_2| \ge \Big(1-\frac1{n}\Big)(|A_1|+|A_2|).
\end{equation}

In view of~\refe{0410}, it suffices to show that
\begin{equation}\label{e:0610a}
  |A_1+A_2| \ge \Big(1+\frac1{n_1}-\frac3n \Big)\,|A_1|
                                  + \Big(1+\frac1{n_2}-\frac3n \Big)\,|A_2|.
\end{equation}
Assuming for definiteness that $n_1\le n_2$, we distinguish two cases.

If $|A_1|/n_1\le|A_2|/n_2$, then we apply~\refe{sumA1A2}, reducing the
inequality to prove to
  $$ \Big(1-\frac1{n}\Big)(|A_1|+|A_2|)
            \ge \Big(1+\frac1{n_1}-\frac3n \Big)\,|A_1|
                               + \Big(1+\frac1{n_2}-\frac3n \Big)\,|A_2|. $$
This can be rewritten as
  $$ \frac1{n_1}\,|A_1| + \frac1{n_2}\,|A_2|
                                       \le 2\,\frac{|A_1|+|A_2|}{n_1+n_2} $$
and, furthermore, as
  $$ (n_1-n_2)\Big(\frac{|A_1|}{n_1}-\frac{|A_2|}{n_2}\Big) \ge 0, $$
which is true by our present assumptions $n_1\le n_2$ and $|A_1|/n_1\le
|A_2|/n_2$.

Assume now that, in addition to $n_1\le n_2$, we have
\begin{equation}\label{e:0510supp}
  |A_1|/n_1 > |A_2|/n_2.
\end{equation}
By Corollary~\refc{BPdist} (applied with $B=A_1$ and $C=A_2$),
  $$ |A_1+A_2| \ge \frac{n-1}{n_1}\,|A_1|. $$
Substituting this estimate into~\refe{0610a}, we see that it suffices to
prove that
  $$ \frac{n-1}{n_1}\,|A_1| \ge \Big(1+\frac1{n_1}-\frac3n \Big)\,|A_1|
                               + \Big(1+\frac1{n_2}-\frac3n \Big)\,|A_2|; $$
that is,
  $$ \Big(\frac{n-2}{n_1}-1+\frac3n\Big)\,|A_1|
                             \ge \Big(1+\frac1{n_2}-\frac3n \Big)\,|A_2|. $$
In view of~\refe{0510supp}, this will follow from
  $$ \Big(\frac{n-2}{n_1}-1+\frac3n\Big)\,n_1
                                \ge \Big(1+\frac1{n_2}-\frac3n \Big)\,n_2 $$
which is easily verified to hold (as an equality, in fact).
\end{proof}

\section{Proof of Theorem~\reft{small}}\label{s:proof}

Recall that we have a finite subset $A$ of the abelian group $G=\Z\oplus H$,
where $H\le G$ is finite. We assume that $|2A|<3(1-\frac1n)|A|$, where $n$ is
the number of elements in the image of $A$ under the projection $G\to\Z$
along $H$, and we want to show that there exist an arithmetic progression
$P\seq G$ and a subgroup $K\le H$ such that $A\seq P+K$ and
$(|P|-1)|K|\le|2A|-|A|$. As shown in Section~\refs{intro}, the estimates
$3\le |P|\le(\tau-1)n+1$ and the equality $|P+K|=|P||K|$ follow automatically
and we disregard them for the rest of the proof. Here $\tau$ is defined by
$|2A|=\tau|A|$, so that $\tau<3(1-\frac1n)$.

Let $\psi\colon G\to\Z$ be the projection mentioned in the previous
paragraph. Without loss of generality we assume that $0\in A$ and
$\min\psi(A)=0$, and we let $l:=\max\psi(A)$; thus, $A\cap(z+H)=\est$ for
$z<0$ and also for $z>l$, while the sets $A_0:=A\cap H$ and $A_l:=A\cap(l+H)$
are nonempty.

Fix arbitrarily an element $\del\in A_l$, and let $A^*:=A_0\cap(A_l-\del)$
and $\sig:=|A_0|+|A_l|$. Notice that $0\in A^\ast$, $\sig\ge 2|A^*|$, and
$|A_0\cup(A_l-\del)|=\sig-|A^*|$.

For a subgroup $L\le G$, by $\phi_L$ we denote the canonical homomorphism of
$G$ onto the quotient group $G/L$. Let $\Del:=\<\del\>\le G$. We adopt a
special notation for the homomorphism $\phi_\Del$, which is particularly
important for our argument: whenever $s$ denotes an element of $G$, by $\os$
we denote the image of $s$ under $\phi_\Del$, and similarly for sets:
$\oS=\phi_\Del(S),\ S\seq G$. Thus, for instance, $\oA=\phi_\Del(A)$ and
$\o{2A}=\phi_\Del(2A)=2\oA$.

To make the proof easier to follow, we split it into several parts.

\subsection{Deficiency and the induction framework}\label{s:def}
We use induction on $|H|$, the base case $|H|=1$ being Freiman's
$(3n-3)$-theorem (see Section~\refs{intro}). Suppose that $|H|\ge 2$.

Given a subset $S\seq G$ and a subgroup $L\le G$, both finite, we define the
\emph{deficiency} of $S$ on a coset $g+L\seq G$ by
  $$ \dfc(S,g+L) :=
       \begin{cases}
         |(g+L)\stm S|\ &\text{if}\ S\cap(g+L)\ne\est, \\
         0            \ &\text{if}\ S\cap(g+L)=\est;
       \end{cases} $$
notice that in the first case we can also write
  $\dfc(S,g+L)=|L|-|(g+L)\cap S|$.
The \emph{total deficiency} of $S$ with respect to $L$ is
  $$ \Dfc(S,L) := |(S+L)\stm S|; $$
equivalently,
  $$ \Dfc(S,L) = \sum_{g+L} \dfc(S,g+L), $$
where the sum extends over all $L$-cosets having a nonempty intersection with
$S$.

Suppose that $L\le H$ is a nonzero finite subgroup with
\begin{equation}\label{e:0201a}
 \Dfc(2A,L) \le \Dfc(A,L).
\end{equation}
Then, letting $T:=3(1-\frac1n)$,
\begin{equation*}
  |2(A+L)| \le |A+L| + |2A| - |A| < |A+L| + (T-1)|A| \le T|A+L|;
\end{equation*}
that is, writing $\tG:=G/L\cong(H/L)\oplus\Z$, $\tA:=\phi_L(A)$, and
$\widetilde{2A}:=\phi_L(2A)$, we have $|2\tA|<3(1-\frac1n)|\tA|$. Applying
the induction hypothesis to the subset $\tA\seq\tG$, we conclude that there
are an arithmetic progression $\tP\seq\tG$ and a subgroup $\tK\le\tH:=H/L$
such that $\tA\seq\tP+\tK$ and $(|\tP|-1)|\tK|\le|2\tA|-|\tA|$. Let
$K:=\phi_L^{-1}(\tK)$; thus, $L\le K\le H$ and $|K|=|L||\tK|$. Also, it is
easily seen that $\phi_L^{-1}(\tP)=P+L$ where $P\seq G$ is an arithmetic
progression with $|P|\le|\tP|$. From $\tA\seq\tP+\tK$ we derive then that
$A\seq P+K$, and from $(|\tP|-1)|\tK|\le|2\tA|-|\tA|$ we get
\begin{multline*}
  (|P|-1)|K| \le (|\tP|-1)|\tK||L| \le (|2\tA|-|\tA|)|L| \\
                  = |2A+L| - |A+L| = |2A|+\Dfc(2A,L)
                                            - |A| - \Dfc(A,L) \le |2A|-|A|,
\end{multline*}
completing the induction step.

Of particular interest is the situation where $L$ is a nonzero, finite
subgroup satisfying
\begin{equation}\label{e:0201b}
  \Dfc(A,L)\le |L|-1.
\end{equation}
Let in this case $m$ denote the number of $L$-cosets on which $A$ has
positive deficiency, and fix $a_1\longc a_m\in A$ such that $a_1+L\longc
a_m+L$ list all these cosets. It follows easily from~\refe{0201b} that there
is at most one pair of indices $1\le i\le j\le m$ such that
$\dfc(A,a_i+L)+\dfc(A,a_j+L)\ge|L|$, and if such a pair exists, then in fact
$i=j$. By the pigeonhole principle, we have then $\dfc(2A,g+L)=0$ for every
coset $g+L$, with the possible exception of one single $L$-coset which is
then of the form $2a+L$, with some $a\in A$. This yields
  $$ \Dfc(2A,L) = \dfc(2A,2a+L) \le \dfc(A,a+L) \le \Dfc(A,L). $$
Clearly, the resulting estimate
  $$ \Dfc(2A,L) \le \Dfc(A,L) $$
remains valid also if there are no exceptional $L$-cosets.

Thus, once we are able to find a nonzero finite subgroup $L<H$ satisfying
either~\refe{0201a} or~\refe{0201b}, we can complete the proof applying the
induction hypothesis.

As a result, we can assume that for any nonempty subsets $A',A''\seq A$ with
$A=A'\cup A''$,
\begin{equation*}\label{e:0701a}
  |A'+A''|\ge|A'|+|A''|-1;
\end{equation*}
for if this fails to hold, then letting $L:=\pi(A'+A'')$, by Kneser's theorem
we have $|L|\ge 2$ and $|A'+L|+|A''+L|-|L|=|A'+A''|\le |A'|+|A''|-2$, whence
  $$ \Dfc(A,L) \le \Dfc(A',L)+\Dfc(A'',L) \le |L|-2 $$
(for the first inequality, notice that
$\dfc(A,g+L)\le\dfc(A',g+L)+\dfc(A'',g+L)$ for any coset $g+L$, which follows
from the assumption $A',A''\seq A=A'\cup A''$).

In particular, we assume that $|A+S|\ge|A|+|S|-1$ for any nonempty subset
$S\seq A$.

As an important special case,
\begin{equation}\label{e:3112}
  |A+A^*| \ge |A| + |A^*| - 1.
\end{equation}

\ifthenelse{\equal{\showmorestuff}{yes}}%
{{\color{brown}%
\begin{remark}
Similarly, if there exists a finite, nonempty subset $S\seq G$ with
$|A+S|<|A|+|S|-1$, then letting $L:=\pi(A+S)$, by Kneser's theorem we have
$|L|\ge 2$ and $|A+L|+|S+L|-|L|=|A+S|\le|A|+|S|-2$, whence
$\Dfc(A,L)=|A+L|-|A|\le|L|-2$. Therefore, we can in fact assume that
$|A+S|\ge|A|+|S|-1$ holds for any finite, nonempty subset $S\seq G$, not
necessarily contained in $A$. This observation, however, does not seem to be
particularly helpful.
\end{remark}}
}{} 

\subsection{The set $\oA$ has small doubling}\label{wraparound}
The quantity $|A^*|$ can be interpreted as the number of representations of
$\del$ as a difference of two elements of $A$. Generally, for a set $S\seq G$
and an element $g\in G$, denote by $r_S(g)$ the number of representations of
$g$ as a difference of two elements of $S$; thus, for instance,
$|A^*|=r_A(\del)$. Clearly, every $\Del$-coset intersects $A$ by at most two
elements, and if the intersection contains \emph{exactly} two elements, then
the two elements differ by $\del$. It follows that
\begin{equation}\label{e:notused}
  |A| = |\oA| + r_A(\del) = |\oA| + |A^*|.
\end{equation}
Similarly, since $\os_1=\os_2$ for any $s_1,s_2\in 2A$ with $s_2-s_1=\del$,
we have $|2A|\ge|2\oA|+r_{2A}(\del)$. Furthermore,
$r_{2A}(\del)\ge|A+A^\ast|$ as to any $a\in A$ and $a^\ast\in A^\ast$ there
corresponds the representation $((a^\ast+\del)+a)-(a^\ast+a)=\del$, and the
sum $a+a^*$ is uniquely determined by this representation. Therefore,
\begin{equation}\label{e:0101a}
  |2A| \ge |2\oA| + |A+A^*|.
\end{equation}

\ifthenelse{\equal{\showmorestuff}{yes}}%
{{\color{brown}
\begin{remark}
Yet another potentially useful observation in this spirit is as follows.
Suppose that $K\le H$ and $F\le G$ are subgroups such that $\oK=\oF$ and $K$
is finite. (The actual subgroups of interest are defined below; presently, we
do not need to know what they are.) Then
\begin{multline*}
  |A+K| = |\oA+\oK| + r_{A+K}(\del) \\
        = |\oA+\oF| + |(A_0+K)\cap(A_l-\del+K)| \ge |\oA+\oF| + |A^*+K|,
\end{multline*}
and it follows that
  $$ \Dfc(A,K) \ge \Dfc(\oA,\oF) + \Dfc(A^*,K). $$
\end{remark}
}}{} 

\ifthenelse{\equal{\showmorestuff}{yes}}%
{{\color{brown}%
\begin{remark}
In general, for any finite subset $S\seq G$ we have $|S|\ge|\oS|+r_S(\del)$.
Applying this with $S=2A$ and observing that
  $$ r_{2A}(\del)=|(2A)\cap(2A+\del)| \ge |A+(A\cap(A+\del))| = |A+A^*|, $$
we get $|2A|\ge|2\oA|+|A|+|A^*|-1$. On the other hand, from $|\oF|=|\oK|=|K|$
it follows that
  $$ |2\oA| = 2|\oA+\oF|-|\oF| \ge 2|\oA|-|K| = 2|A|-2|A^\ast|-|K|. $$
Combining the two last estimates,
\begin{equation*}\label{e:basic}
  |2A|-|A|-|A^*|+1 \ge |2\oA| \ge 2|A|-2|A^\ast|-|K|,
\end{equation*}
whence $(3-\tau)|A|\le|A^*|+|K|+1$.
\end{remark}
}}{} 

We now claim that
\begin{equation}\label{e:2oAisperiodic}
  |2\oA| < 2|\oA| - 1.
\end{equation}
In view of $|2\oA|\le |2A|-|A+A^*|\le\tau|A| - |A|-|A^*|+1$ and
$|\oA|=|A|-|A^*|$ (cf.~\refe{0101a}, \refe{3112}, and \refe{notused}), it
suffices to show that
  $$ \tau |A| - |A|-|A^*| + 1 < 2|A| - 2|A^*| - 1; $$
that is,
\begin{equation}\label{e:2oAtoshow}
  (3-\tau)|A| > |A^*| + 2.
\end{equation}
To this end we notice that, by~\refe{0301b} and in view of
$|A^*|\le\min\{|A_0|,|A_l|\}$,
  $$ (3-\tau)(\tau|A|+|A^*|) > 3(|A_0|+|A_l|) \ge 6|A^*|. $$
Consequently,
  $$ (3-\tau) |A| > \left(\frac3{\tau}+1\right) |A^*| > 2|A^*|, $$
which proves~\refe{2oAtoshow} in the case where $|A^*|\ge2$. In the remaining
case $|A^*|=1$, we obtain~\refe{2oAtoshow} as an immediate corollary of
$|A|\ge n$ and $\tau<3\left(1-\frac1n\right)$.

Thus, \refe{2oAisperiodic} is established, and from Kneser's theorem it
follows that the period $\oF:=\pi(2\oA)$ is a nonzero subgroup of the
quotient group $G/\Del$, and also, in view of
 $2|\oA+\oF|-|\oF|=|2\oA|\le 2|\oA|-2$, that
\begin{equation}\label{e:def}
  \Dfc(\oA,\oF) \le \frac12\,|\oF|-1.
\end{equation}
We let $F:=\phi_\Del^{-1}(\oF)$, so that $\oF=\phi_\Del(F)$ and
 $\Del\le F\le G$.
\ifthenelse{\equal{\showmorestuff}{yes}}%
{
\begin{remark}
The relation $2\oA=2\oA+\oF$ can be equivalently rewritten as $2A+\Del=2A+F$.
This shows that if an $F$-coset, say $\cF$, intersects $2A$, then indeed
every $\Del$-coset contained in $\cF$ intersects $2A$.
\end{remark}
}{} 

Observing that $0\in A$ implies $\oA+\oF\seq 2\oA+\oF=2\oA$, we denote by $N$
the number of $\oF$-cosets contained in $2\oA$, but not in $\oA+\oF$; that
is,
  $$ N = (|2\oA|-|\oA+\oF|)/|\oF|. $$
Combining $|2\oA|-|\oA+\oF|=N|\oF|$ and $|2\oA|=2|\oA+\oF|-|\oF|$, we get
\begin{equation}\label{e:2501a}
  |\oA+\oF|=(N+1)|\oF|\ \text{and}\ |2\oA|=(2N+1)|\oF|.
\end{equation}

Let $K:=F\cap H$.

\subsection{The case where $N=0$}
If $N=0$, then $\oA+\oF=2\oA$. Adding $\oA$ to both sides we get
$2\oA=2\oA+\oA$, showing that $\oA\seq\pi(2\oA)=\oF$. Combining this with
$\oA+\oF=2\oA$, we conclude that $2\oA=\oF$. Thus, $2A+\Del=F$ and,
by Lemma~\refl{module},
\begin{equation}\label{e:2101a}
  A \seq 2A + \Del = F = \<g\> + K
\end{equation}
with some $g\in G$. Notice that $g\notin H$, as otherwise we would have
$A\seq H$, and hence $n=1$.

Let $P:=\<g\>\cap\psi^{-1}\big([0,l]\big)$, so that $A\seq P+K$. Since
$\psi^{-1}\big([0,l)\big)$ contains exactly one representative out of every
$\Del$-coset, we also have
\begin{align*}
  |2\oA|
      &= |\phi_\Del(2A)| \\
      &= |\phi_\Del(2A+\Del)| \\
      &= \big|(2A+\Del)\cap\psi^{-1}\big([0,l)\big)\big| \\
      &= \big|\big(\<g\>+K\big)
                       \cap\psi^{-1}\big([0,l)\big)\big| \\
      &= \big|\<g\>\cap\psi^{-1}\big([0,l)\big)\big||K| \\
      &= \big|\<g\>\cap\psi^{-1}\big([0,l]\big)\big||K|
                 - \big|\<g\>\cap\psi^{-1}\big(l\big) ||K| \\
      &= (|P|-1) |K|,
\end{align*}
the middle equality following from~\refe{2101a}, and the last equality from
  $$ \est \ne A\cap \psi^{-1}(l)
            \seq (\<g\>+K) \cap \psi^{-1}(l)
                = (\<g\> \cap \psi^{-1}(l)) + K $$
and the resulting $\<g\> \cap \psi^{-1}(l)\ne\est$.
 Consequently,~\refe{0101a} yields
$(|P|-1)|K|\le|2A|-|A|$, completing the proof in the present case.

We thus assume for the remaining part of the argument t=hat $N>0$; that is
\begin{equation*}\label{e:oAoF2}
  \oA+\oF\subsetneq 2\oA.
\end{equation*}
Therefore, $2\oA$ is \emph{not} a subgroup (if it were, we would have
$\oF=\pi(2\oA)=2\oA$ implying $\oA+\oF\supseteq 2\oA$).

\subsection{The case where $N=1$}
If $N=1$, then $\oA+\oF$ is a union of exactly two $\oF$-cosets, and $2\oA$
is a union of exactly three $\oF$-cosets. Since $0\in A$, we derive that
$A=A_1\cup(g+A_2)$, where $A_1,A_2\seq F$ are nonempty and finite, and where
$g\in G$ satisfies $2g\notin F$, as a result of $2\oA$ being a union of three
$\oF$-cosets. Write $n_i:=|\psi^{-1}(A_i)|,\ i\in\{1,2\}$, so that
$n:=|\psi^{-1}(A)|\le n_1+n_2$. By Proposition~\refp{N=1}, we have then
\begin{align*}
  |2A| &=   |2A_1|+|A_1+A_2|+|2A_2| \\
       &\ge 3\Big(1-\frac1{n_1+n_2}\Big)(|A_1|+|A_2|) \\
       &\ge 3\Big(1-\frac1n\Big)|A|,
\end{align*}
a contradiction.

\medskip
Let $\oH:=\phi_\Del(H)$ and, following our standard convention, write
$\oK:=\phi_\Del(K)$. We split the remaining case $N\ge 2$ into two further
subcases: that where $\oK$ is a proper subgroup of $\oF$ (which, by
Corollary~\refc{FGHDel}, is equivalent to any of $\oF\not\le\oH$ and
$F\not\le H+\Del$), and that where $\oK=\oF$ (equivalently, $\oF\le\oH$,
$F\le H+\Del$, or $F=K\oplus\Del$).

\subsection{The case where $N\ge 2$ and $\oK\lneqq\oF$.}
We show that in this case
\begin{equation}\label{e:2show}
  |2A\stm A|\ge 2|\oA|;
\end{equation}
in view of~\refe{notused} and~\refe{1912}, this will give
  $$ |2A|-|A| \ge 2|\oA| = 2|A|-2|A^*| > 2|A|-(3-\tau)|A| = (\tau-1)|A|, $$
a contradiction.

To prove~\refe{2show}, we partition the elements $s\in2A\stm A$ into two
groups, according to whether $\os=\phi_\Del(s)$ lies in $\oA+\oF$.

For the first group we have the estimate
\begin{equation*}\label{e:oAplusoF}
  |\{ s\in 2A\stm A\colon \os\in\oA+\oF \}| \ge |\oA+\oF|;
\end{equation*}
for, $\oA+\oF\seq2\oA$ shows that for every element $\os\in\oA+\oF$, the set
$\{s\in 2A\colon\phi_\Del(s)=\o{s}\}$ is nonempty, and the (unique) element
of this set with the largest value of $\psi(s)$ does not lie in $A$ as $s\in
A$ implies $s+\del\in 2A$, because of $\del\in A$. (This argument shows that,
indeed, for any subset $\oS\seq2\oA$ there are at least $|\oS|$ elements
$s\in 2A\stm A$ such that $\os\in\oS$.)

To complete the treatment of the case $\oK\lneqq\oF$, we show that
  $$ T := |\{ s\in 2A\stm A\colon \os\notin\oA+\oF \}|
                                                 \ge 2|\oA| - |\oA+\oF|. $$
(Notice that the trivial estimate would be $T\ge|2\oA|-|\oA+\oF|$.)

The set $(2\oA)\stm(\oA+\oF)$ is a union of $\oF$-cosets, and we find
$a_1\longc a_N,b_1\longc b_N\in A$ such that the cosets in question are
$\oa_i+\ob_i+\oF$, $i\in[1,N]$.

Let
  $$ A_i :=A\cap(a_i+F)\ \text{and}
                                 \  B_i := A\cap(b_i+F),\quad i\in[1,N]. $$
By Corollary~\refc{FGHDel} we have $\oA_i=\oA\cap(\oa_i+\oF)$ and
$\oB_i=\oA\cap(\ob_i+\oF)$, and it follows that
\begin{equation}\label{e:Ailarge}
  |A_i| \ge |\oA_i| = |\oa_i+\oF| - |(\oa_i+\oF)\stm\oA|
                              = |\oF| - \dfc(\oA,\oa_i+\oF),\quad i\in[1,N]
\end{equation}
and, similarly,
\begin{equation}\label{e:Bilarge}
  |B_i| \ge |\oB_i| = |\ob_i+\oF| - |(\ob_i+\oF)\stm\oA|
                              = |\oF| - \dfc(\oA,\ob_i+\oF),\quad i\in[1,N].
\end{equation}
Since
 $\phi_\Del(A_i+B_i)=\oA_i+\oB_i\seq\oa_i+\ob_i+\oF\seq(2\oA)\stm(\oA+\oF)$
by the choice of $a_i$ and $b_i$, we have
\begin{equation}\label{e:2101b}
  T = \sum_{i=1}^N
            |\{ s\in 2A\stm A\colon \os\in\oa_i+\ob_i+\oF \}|
                                                 \ge \sum_{i=1}^N |A_i+B_i|.
\end{equation}

By Lemma~\refl{pairs} as applied to the subset $\tA:=(\oA+\oF)/\oF$ of the
quotient group $\overline{G}/\oF$, we can assume that each $\oF$-coset from
$\oA+\oF$ appears among the $2N$ cosets $\oa_1+\oF\longc\ob_N+\oF$ at most
$N$ times, except if  there is a subgroup $\tL\le\oG/\oF$ and an element
$\tc\in\oG/\oF$ with $2\tc\notin\tL$ such that either $\tA=\tL\cup\{\tc\}$,
or $\tA=(\tc+\tL)\cup\{0\}$. In this exceptional situation $A$ meets exactly
two cosets of the subgroup $L=\phi_F^{-1}(\tL)$, while $2A$ meets exactly
three cosets of this subgroup. As a result, we can apply
Proposition~\refp{N=1}, exactly as in the case $N=1$ considered above, to get
$|2A|\ge3\Big(1-\frac1n\Big)|A|$.

Addressing now case i) of Lemma~\refl{pairs}, assume that each $\oF$-coset
from $\oA+\oF$ appears among $\oa_1+\oF\longc\ob_N+\oF$ at most $N$ times.

Since $A_i+B_i$ is contained in an $F$-coset, we have $\pi(A_i+B_i)\le F$,
and since $A_i+B_i$ is finite, $\pi(A_i+B_i)\le H$; as a result,
$\pi(A_i+B_i)\le F\cap H=K$. Consequently, by~\refe{2101b}, Kneser's theorem,
\refe{Ailarge}, and~\refe{Bilarge},
  $$ T \ge 2N|\oF|
         - \sum_{i=1}^N (\dfc(\oA,\oa_i+\oF)+\dfc(\oA,\ob_i+\oF))-|K|N. $$

Recalling that  each $\oF$-coset from $\oA+\oF$ appears at most $N$ times
among $\oa_1+\oF\longc\ob_N+\oF$, we get
  $$ T \ge 2N|\oF| - N\Dfc(\oA,\oF) - |K|N
                     \ge \left( \frac32\,|\oF| - \Dfc(\oA,\oF) \right) N $$
(as $\oK\lneq\oF$ yields $|K|=|\oK|\le\frac12|\oF|$). Therefore,
by~\refe{def}, \refe{2501a}, and the definition of the total deficiency,
\begin{align*}
  T &\ge (N+1)|\oF|-2\Dfc(\oA,\oF)
                    + (N-2)\left( \frac12\,|\oF|-\Dfc(\oA,\oF) \right) \\
    &\ge (N+1)|\oF|-2\Dfc(\oA,\oF) \\
    &=   |\oA| - \Dfc(\oA,\oF) \\
    &=   2|\oA| - |\oA+\oF|.
\end{align*}
As explained above, this leads to a contradiction.

\subsection{The case where $N\ge 2$ and $\oK=\oF$.}
As shown above, in this case $\oF\le\oH$, $F\le H+\Del$, and $F=K\oplus\Del$;
notice that this implies $|\oF|=|\oK|=|K|$.

Let $A^\circ:=A\stm(A_0\cup A_l)$; loosely speaking, $A^\circ$ is the
``middle part'' of $A$.

\begin{claim}\label{m:1101c}
We have $A_0\seq K$ and $A_l\seq\del+K$; that is, each of the sets $A_0$ and
$A_l$ is contained in a single $K$-coset.
\end{claim}

\begin{proof}
From~\refe{0101a}, Kneser's theorem, \refe{3112}, and~\refe{notused}, we have
\begin{multline*}
  |2A| \ge (2|\oA+\oF| - |\oF|) + (|A|+|A^*|-1) \\
                   \ge 2|\oA|-|\oF| + |A|+|A^*|-1 = 3|A| - |A^*| - |\oF| -1.
\end{multline*}
Combining this estimate with~\refe{1001a}, we get
\begin{equation}\label{e:2101c}
  \sig\le |A^*|+|\oF|.
\end{equation}

On the other hand,
  $$ |H\cap(A+\Del)| \ge |\phi_\Del(H\cap(A+\Del))|
                   = |\phi_\Del(H)\cap\phi_\Del(A+\Del)| = |\oH\cap\oA| $$
by Lemma~\refl{2101}. Observing that the left-hand side is
  $$ |(H\cap A) \cup (H\cap(A-\del))| = |A_0\cup(A_l-\del)|
                                                         = \sig - |A^*|, $$
and using~\refe{2101c}, we obtain
\begin{equation*}\label{e:2912a}
  |\oH\cap\oA| \le \sig - |A^*| \le |\oF|.
\end{equation*}

Assuming now for a contradiction that, say, $A_0$ intersects more than one
$K$-coset, fix $a_1,a_2\in A_0$ which are distinct modulo $K$. Since
$\oa_1,\oa_2\in\oH$ are then distinct modulo $\oK=\oF$, in view of~\refe{def}
and the assumption $\oF\le\oH$ we get
\begin{align*}
  \frac12|\oF| &> \Dfc(\oA,\oF) \\
    &\ge \dfc(\oA,\oa_1+\oF) + \dfc(\oA,\oa_2+\oF) \\
    &=   2|\oF| - (|(\oa_1+\oF)\cap\oA|+|(\oa_2+\oF)\cap\oA|) \\
    &\ge 2|\oF| - |(\oH+\oF)\cap\oA| \\
    &=   2|\oF| - |\oH\cap\oA| \\
    &\ge |\oF|,
\end{align*}
the contradiction sought.
\end{proof}

\begin{claim}\label{m:new}
We have $2A^\circ+K=2A^\circ$. Moreover, if $|A_0|\ge|A_l|$, then
$A^\circ+K\seq 2A$, and if $|A_l|\ge|A_0|$, then $A^\circ+\del+K\seq 2A$.
\end{claim}

\begin{proof}
To prove the first assertion, we fix $a_1,a_2\in A^\circ$ and show that
$a_1+a_2+K\seq 2A^\circ$. For $i\in\{1,2\}$, let $A_i:=(a_i+F)\cap A$; notice
that $A_i\seq a_i+F=a_i+K+\Del$ whence, indeed, $A_i\seq a_i+K$. Write
$S:=A_1+A_2\seq a_1+a_2+K$ so that $\oS=\oA_1+\oA_2=\oa_1+\oa_2+\oF$ in view
of~\refe{def}. As a result, $|S|\ge|\oS|=|\oa_1+\oa_2+\oF|=|\oF|=|\oK|=|K|$,
leading to $S=a_1+a_2+K$; thus, $a_1+a_2+K=A_1+A_2\seq 2A^\circ$.

Addressing the second assertion, we fix $a^\circ\in A^\circ$ and show that
then either $a^\circ+K\seq 2A$, or $a^\circ+\del+K\seq 2A$, according to the
relation between $|A_0|$ and $|A_l|$. Write $B_0:=A\cap F$ and
$B^\circ:=A\cap(a^\circ+F)$; equivalently, $B_0=A_0\cup A_l$ by
Claim~\refm{1101c}, and $B^\circ=A\cap(a^\circ+K)$. Letting $S:=B_0+B^\circ$,
in view of $B_0\seq F$ and $B^\circ\seq a^\circ+F$ we have then $S\seq
2A\cap(a^\circ+F)$ and $\oS=\oB_0+\oB^\circ$. Furthermore, from
\begin{align*}
  \dfc(\oA,\oF) &= |\oF|-|\oA\cap\oF| = |\oF| - |\oB_0|
  \intertext{and}
  \dfc(\oA,\oa^\circ+\oF) &= |\oF| - |\oA\cap(\oa^\circ+\oF)|
                                                     = |\oF| - |\oB^\circ|
\end{align*}
recalling~\refe{def} we get
  $$ |\oB_0|+|\oB^\circ| = 2|\oF| - (\dfc(\oA,\oF)+\dfc(\oA,\oa^\circ+\oF))
          \ge 2|\oF| - \Dfc(\oA,\oF) > \frac32\,|\oF|. $$
From $B_0=A_0\cup A_l$ we now derive
  $$ |A_0|+|A_l| + |B^\circ| \ge |B_0|+|B^\circ|
                             \ge |\oB_0| + |\oB^\circ| > \frac32\,|\oF|. $$
Also, we have
  $$ |B^\circ| \ge |\oB^\circ| = |\oA\cap(\oa^\circ+\oF)|
        = |\oF| - \dfc(\oA,\oa^\circ+\oF)
                             \ge |\oF| - \Dfc(\oA,\oF) > \frac12\,|\oF|. $$
Therefore,
  $$ \max\{|A_0|,|A_l|\}+|B^\circ|
           \ge \frac12(|A_0|+|A_l|+|B^\circ|) + \frac12\,|B^\circ|
             > \frac34\,|\oF| + \frac14\,|\oF| = |\oF| = |K|. $$
Since $B^\circ\seq a^\circ+K$, $A_0\seq K$, and $A_l\seq\del+K$, from the
pigeonhole principle we conclude that if $|A_0|>|A_l|$, then
$A_0+B^\circ=a^\circ+K$, while if $|A_l|\ge|A_0|$, then
$A_l+B^\circ=a^\circ+\del+K$. The assertion follows in view of
$A_0+B^\circ\seq 2A$ and $A_l+B^\circ\seq 2A$.
\end{proof}

We can, eventually, complete the proof. Assuming $|A_0|\le|A_l|$ for
definiteness, by Claim~\refm{new} we have $2A^\circ+K\seq 2A$ and also
$A^\circ+A_l+K\seq 2A$; that is, the set $2A$ has zero deficiency on all
$K$-cosets with the possible exception of the cosets contained in $A_0+A+K$;
that is, cosets of the form $a+K$ with $a\in A$. On the other hand, in view
of
  $$ A\cap(a+K) + A_0 \seq 2A\cap(a+K) $$
and $|2A\cap(a+K)|\ge|A\cap(a+K)|$ resulting from it, we have
  $$ \dfc(2A,a+K) \le \dfc(A,a+K). $$
Taking the sum over the elements $a\in A$ representing the $K$-cosets
contained in $A$ we get
  $$ \Dfc(A,K) = \sum_{a} \dfc(A,a+K)
                  \ge \sum_{a} \dfc(2A,a+K) = \Dfc(2A,K). $$
As noticed in Section~\refs{def}, this completes the proof by appealing to
the induction.

\bigskip
\end{document}